\documentclass[11pt]{article}

\usepackage{graphicx,algorithmic,algorithm}
\usepackage{enumitem}
\usepackage{tikz}
\usepackage{lineno}
\usepackage[nomessages]{fp}
\usetikzlibrary{shapes}
\usepackage[a4paper, margin=2cm]{geometry}
\usepackage{hyperref}
\usepackage{subcaption}
\hypersetup{
    colorlinks,
    linkcolor = {red!60!black},
    citecolor = {green!60!black},
    urlcolor  = {blue!60!black},
    pdftitle  = {Unavoidable pivot-minors in graphs of large rank-depth},
    pdfauthor = {Jungho Ahn, Kevin Hendrey, O-joung Kwon, and Sang-il Oum},
}

\usepackage{todonotes}
\usepackage{authblk}
\usepackage{amsthm,amssymb,amsmath}
\usepackage{mathabx}
\usepackage{thmtools}
\usepackage[capitalise,nameinlink,noabbrev]{cleveref}

\newtheorem{theorem}{Theorem}[section]
\newtheorem{lemma}[theorem]{Lemma}
\newtheorem{proposition}[theorem]{Proposition}
\newtheorem{corollary}[theorem]{Corollary} 

\crefname{observation}{Observation}{Observations}
\newtheorem{conjecture}[theorem]{Conjecture}

\crefname{claim}{Claim}{Claims}

\newcommand\abs[1]{\lvert #1\rvert}

\newcommand\tri{\boxslash}

\newcommand{\dist}{\mathrm{dist}}

\def\K_#1{{K_{#1}}}
\def\S_#1{\overline{K_{#1}}}

\newcommand\extrafootertext[1]{%
    \bgroup
    \renewcommand\thefootnote{\fnsymbol{footnote}}%
    \renewcommand\thempfootnote{\fnsymbol{mpfootnote}}%
    \footnotetext[0]{#1}%
    \egroup
}

\begin{document}
\title{Unavoidable pivot-minors in graphs of large rank-depth}

\author[1]{Jungho Ahn\thanks{Supported by the KIAS Individual Grant (CG095301) at Korea Institute for Advanced Study and Leverhulme Trust Research Project Grant RPG-2024-182.}}
\author[2]{Kevin Hendrey\thanks{Supported by the Institute for Basic Science (IBS-R029-C1).}\thanks{Supported by the Australian Research Council.}}
\author[3,4]{O-joung Kwon\textsuperscript{\textdagger}\thanks{Supported by the National Research Foundation of Korea (NRF) grant funded by the Ministry of Science and ICT (No. RS-2023-00211670).}}
\author[4]{Sang-il Oum\textsuperscript{\textdagger}}
\affil[1]{Department of Computer Science, Durham University, Durham, UK}
\affil[2]{School of Mathematics, Monash University, Melbourne, Australia}
\affil[3]{Department of Mathematics, Hanyang University, Seoul, South~Korea}
\affil[4]{Discrete Mathematics Group, Institute~for~Basic~Science~(IBS), Daejeon,~South~Korea.}
\affil[ ]{\small\textit{Email addresses:}
\texttt{jungho.ahn@durham.ac.uk} (Ahn),
\texttt{kevin.hendrey1@monash.edu} (Hendrey),
\texttt{ojoungkwon@hanyang.ac.kr} (Kwon),
\texttt{sangil@ibs.re.kr} (Oum).
}

\date\today
\maketitle

\begin{abstract}
    Shrub-depth and rank-depth are related graph parameters that are dense analogs of tree-depth.
    We prove that for every positive integer $t$, every graph of sufficiently large rank-depth contains a pivot-minor isomorphic to a path on $t$ vertices or a graph consisting of two disjoint cliques of size $t$ joined by a half graph.
    This answers an open problem raised by Kwon, McCarty, Oum, and Wollan in 2021.
\end{abstract}

\section{Introduction}

Tree-width, path-width, and tree-depth of graphs are well-known important parameters of graphs that have been developed with the study of graph minors, notably during the graph minors project of Robertson and Seymour from 1980s, see \cite{RS1985}.
For these graph parameters, we have precise characterizations of graph classes of bounded parameters in terms of forbidden minors.
We say a class of graphs is \emph{minor-closed} if every minor of a graph in the class is also contained in the class.

\begin{theorem}[Robertson and Seymour]
    Let $\mathcal C$ be a minor-closed class of graphs.
    Then the following hold.
    \begin{enumerate}[label=\rm (\roman*)]
        \item $\mathcal C$ has bounded tree-width if and only if it does not contain some planar graph~\cite{RS1986}.
        \item $\mathcal C$ has bounded path-width if and only if it does not contain some tree~\cite{RS1983}.
        \item $\mathcal C$ has bounded tree-depth if and only if it does not contain some path graph~\cite[4.4]{RS1985}.
    \end{enumerate}
\end{theorem}
These parameters can be thought of as concrete ways to measure the complexity of minor-closed classes of graphs (which have bounded average degree unless they contain all graphs).
However, since each of these parameters is unbounded on the class of complete graphs, these parameters are not good measures of complexity for hereditary graph classes.
To address this, dense analogs of these parameters for graphs have been introduced.
These dense analogs have the property that any bound on the original parameter for a graph~$G$ translates naturally to a bound on the analogous parameter for both $G$ and its complement, despite the fact that the latter is dense.
Rank-width of graphs is a dense analog of tree-width, which was introduced by Oum and Seymour~\cite{OS2004}.
Linear rank-width of graphs is a dense analog of path-width and is a linearized version of rank-width, see~\cite{AFP2013}.
Rank-depth of graphs is a dense analog of tree-depth, which was introduced by DeVos, Kwon, and Oum~\cite{DKO2019}.
Shrub-depth was introduced earlier by Ganian, Hlin\v{e}n\'{y}, Ne{\v{s}}et{\v{r}}il, Obdr\v{z}\'{a}lek, Ossona de Mendez, and Ramadurai~\cite{GHNOOR2012} as a dense analog of tree-depth.
In fact, shrub-depth is equivalent to rank-depth in the sense that a class of graphs has bounded shrub-depth if and only if it has bounded rank-depth, shown by DeVos, Kwon, and Oum~\cite{DKO2019}.
We are concerned with the rank-depth of graphs and its precise definition will be given in \Cref{sec:preliminaries}.

Vertex-minors and pivot-minors play a similar role to that of minors in the study of dense graphs.
The definitions of vertex-minors and pivot-minors will be given in \Cref{sec:preliminaries}.
For vertex-minors, we have the following theorems and conjectures.
We say a class of graphs is \emph{vertex-minor-closed} if every vertex-minor of a graph in the class is also contained in the class.
A \emph{circle graph} is the intersection graph of chords of a circle.
\begin{theorem}[Geelen, Kwon, McCarty, and Wollan~\cite{GKMW2019}]
    A vertex-minor-closed class of graphs has bounded rank-width if and only if it does not contain some circle graph.
\end{theorem}

\begin{conjecture}[Kant\'e and Kwon~\cite{KK2015}]
    A vertex-minor-closed class of graphs has bounded linear rank-width if and only if it does not contain some tree.
\end{conjecture}

\begin{theorem}[Kwon, McCarty, Oum, and Wollan~\cite{KMOW2019}]\label{thm:vertex-minor-rank-depth}
    A vertex-minor-closed class of graphs has bounded rank-depth if and only if it does not contain some path graph.
\end{theorem}

We say a class of graphs is \emph{pivot-minor-closed} if every pivot-minor of a graph in the class is also contained in the class.
As we shall see clearly from its definition, every pivot-minor of a graph is a vertex-minor of the graph, so every vertex-minor-closed class of graphs is pivot-minor-closed.
There are very close relationships between pivot-minors of bipartite graphs and minors of binary matroids, see~\cite{Oum2004}, which gives a strong motivation to study pivot-minors of graphs.
\Cref{thm:vertex-minor-rank-depth} motivated Kwon, McCarty, Oum, and Wollan~\cite{KMOW2019} to conjecture that
\begin{quote}
    every pivot-minor-closed class of graphs has bounded rank-depth if and only if it does not contain~$P_t$ or $K_t\tri K_t$ for some~$t$, where~$P_t$ is a path on~$t$ vertices and $K_t\tri K_t$ is the graph obtained from two disjoint copies of~$K_t$ by adding a half graph between them.
\end{quote} 
We answer this conjecture in the affirmative by proving the following theorem.

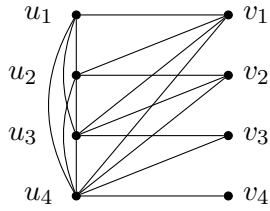
\begin{figure}
    \centering
    \begin{tikzpicture}[scale=0.8]
        \tikzstyle{v}=[circle, draw, solid, fill=black, inner sep=0pt, minimum width=3pt]

        \node[v,label={[xshift=-3pt]left:{$u_1$}}] (u1) at (0,3) {};
        \node[v,label={[xshift=-9pt]left:{$u_2$}}] (u2) at (0,2) {};
        \node[v,label={[xshift=-9pt]left:{$u_3$}}] (u3) at (0,1) {};
        \node[v,label={[xshift=-3pt]left:{$u_4$}}] (u4) at (0,0) {};
        \node[v,label=right:{$v_1$}] (v1) at (2.5,3) {};
        \node[v,label=right:{$v_2$}] (v2) at (2.5,2) {};
        \node[v,label=right:{$v_3$}] (v3) at (2.5,1) {};
        \node[v,label=right:{$v_4$}] (v4) at (2.5,0) {};

        \draw (u1)--(u4);
        \draw (v1)--(u1);
        \draw (v1)--(u2);
        \draw (v1)--(u3);
        \draw (v1)--(u4);
        \draw (v2)--(u2);
        \draw (v2)--(u3);
        \draw (v2)--(u4);
        \draw (v3)--(u3);
        \draw (v3)--(u4);
        \draw (v4)--(u4);

        \draw[bend right=30] (u1) to (u4);
        \draw[bend right=20] (u1) to (u3);
        \draw[bend right=20] (u2) to (u4);
    \end{tikzpicture}
    \caption{The graph $K_4\tri\overline{K_4}$.}
    \label{fig:tri}
\end{figure}

\begin{theorem}\label{thm:main}
    There exists a function $f:\mathbb{N}\to\mathbb{N}$ such that for every $t\in\mathbb{N}$, every graph of rank-depth at least~$f(t)$ has a pivot-minor isomorphic to~$P_t$ or $K_t\tri K_t$.
\end{theorem}

\cref{thm:main} implies \cref{thm:vertex-minor-rank-depth}, because $K_t\tri K_t$ contains $K_{t-1}\tri \overline{K_{t-1}}$ as a vertex-minor, which has $P_t$ as a vertex-minor, see~
\cref{fig:tri} and~\cref{lem:second}.
Another consequence of \cref{thm:main} is about binary matroids.
By using the relationship between pivot-minors of bipartite graphs and minors of binary matroids, we can easily deduce that every binary matroid of sufficiently large branch-depth has the cycle matroid of a large fan graph as a minor.
This has been proved earlier by Gollin, Hendrey, Mayhew, and Oum~\cite{GHMO2021}, more generally for matroids representable over a fixed finite field or quasi-graphic matroids.

One of the key ingredients in the proof of \Cref{thm:main} is the result of M\"{a}hlmann~\cite{Mahlmann25}, which gives a list of unavoidable induced subgraphs in a class graphs of unbounded rank-depth.

This paper is organized as follows.
In \cref{sec:preliminaries}, we introduce terminology from graph theory, including the formal definitions of a pivot-minor and the rank-depth of a graph.
In \cref{sec:flipped graphs}, we introduce the definition of a flipped $mP_n$, which is obtained from the disjoint union of~$m$ copies of~$P_n$ by complementing the edge relation between some sets of vertices.
We show that if~$m$ is sufficiently large, then every flipped $mP_n$ has a pivot-minor isomorphic to the graph obtained from~$P_n$ by complementing the edge relation inside some set of vertices.
In \cref{sec:1-flip}, we show that for a sufficiently large~$n$, the pivot-minor found in \cref{sec:flipped graphs} has~$P_t$ as a pivot-minor, and in \cref{sec:proof}, we prove \Cref{thm:main}.

\section{Preliminaries}\label{sec:preliminaries}

Graphs in this paper are finite, undirected, and simple.
For an integer~$n$, we denote by~$[n]$ the set of positive integers less than or equal to~$n$.
The \emph{symmetric difference} of sets~$X$ and~$Y$, denoted by ${X\triangle Y}$, is the set ${(X\setminus Y)\cup(Y\setminus X)}$.
For a collection~$\mathcal{P}$ of disjoint sets, we denote by~$\bigcup\mathcal{P}$ the set $\bigcup_{X\in\mathcal{P}}X$.
A \emph{coarsening} of~$\mathcal{P}$ is a collection of disjoint non-empty sets that are  unions of members of $\mathcal{P}$.
Note that $\bigcup Q \subseteq \bigcup P$ if $\mathcal Q$ is a coarsening of~$\mathcal{P}$.

For a vertex~$v$ of a graph~$G$, we denote by~$N_G(v)$ the set of neighbors of~$v$ in~$G$ and define~$N_G[v]:=N_G(v)\cup\{v\}$.
For a set $X\subseteq V(G)$, let $N_G[X]:=\bigcup_{v\in X}N_G[v]$ and let $N_G(X):=N_G[X]\setminus X$.
We denote by ${G-X}$ the graph obtained from~$G$ by removing all vertices in~$X$ and all edges incident with vertices in~$X$.
If $X=\{v\}$, then we may write ${G-v}$ for ${G-X}$.
The \emph{subgraph of~$G$ induced on~$X$} is the graph ${G[X]:=G-(V(G)\setminus X)}$.
The \emph{complement of~$G$}, denoted by~$\overline{G}$, is the graph with vertex set $V(G)$ such that distinct vertices are adjacent in~$\overline{G}$ if and only if they are not adjacent in~$G$.
For graphs $G=(V,E)$ and $H=(V',E')$, the \emph{union} of~$G$ and~$H$ is the graph ${G\cup H:=(V\cup V',E\cup E')}$.
We call it the \emph{disjoint union} of~$G$ and~$H$ if~$V$ and~$V'$ are disjoint.

For a tree~$T$, we denote by~$L(T)$ the set of leaves of~$T$.
A \emph{clique} of a graph~$G$ is a set of pairwise adjacent vertices of~$G$, and an \emph{independent set} of~$G$ is a set of pairwise non-adjacent vertices of~$G$.
For sets ${X,X'\subseteq V(G)}$, an \emph{$(X,X')$-path} of~$G$ is a path from a vertex in~$X$ to a vertex in~$X'$ with no internal vertex in ${X\cup X'}$.
We call it an \emph{$X$-path} if ${X=X'}$.
For vertices $u,v$ of~$G$, the \emph{distance} between~$u$ and~$v$ in~$G$, denoted by $\dist_G(u,v)$, is the length of a shortest path between~$u$ and~$v$.
If~$G$ has no such path, then we define $\dist_G(u,v)$ as $+\infty$.
The \emph{radius} of a connected graph~$G$ is the minimum integer~$r$ such that~$G$ has a vertex which is at distance at most~$r$ from every vertex of~$G$.

For positive integers~$m$ and~$n$, we denote by~$K_n$ the complete graph on~$n$ vertices and by~$K_{m,n}$ the complete bipartite graph where one side has~$m$ vertices and the other side has~$n$ vertices.
We also denote by~$P_n$ a path on~$n$ vertices and by $mP_n$ the disjoint union of~$m$ copies of~$P_n$.
The vertex set of~$mP_n$ is $[m]\times[n]$ where $(i,j)$ is the $j$-th vertex of the $i$-th path~$P_n$.
For $i\in[m]$, the \emph{$i$-th row} of~$mP_n$ is the subpath of~$mP_n$ with vertex set $\{i\}\times[n]$.
For $j\in[n]$, the \emph{$j$-th column} of~$mP_n$ is the set ${\{(i,j):i\in[m]\}}$, that is the set of all $j$-th vertices of each component of $mP_n$.
The \emph{column set} of~$mP_n$ is the set of columns of~$mP_n$.

Let~$G$ and~$H$ be vertex-disjoint graphs on~$n$ vertices with vertex orderings ${\{u_1,u_2,\ldots,u_n\}}$ and ${\{v_1,v_2,\ldots,v_n\}}$, respectively.
We denote by $G\tri H$ the graph obtained from $G\cup H$ by adding edges between $V(G)$ and $V(H)$ such that for all $i,j\in[n]$, $u_i$ and~$v_j$ are adjacent in $G\tri H$ if and only if ${i\geq j}$.
We refer to \cref{fig:tri} for an example.

\subsection{Vertex-minors and pivot-minors}

For a vertex~$v$ of a graph~$G$, the graph obtained by applying \emph{local complementation} at~$v$ to~$G$ is
\[
    G\ast v:=(V(G),E(G)\triangle\{xy:x,y\in N_G(v),~x\neq y\}).
\]
For an edge~$uv$ of~$G$, the graph obtained by \emph{pivoting~$uv$} is $G\wedge uv:=G\ast u\ast v\ast u$, which is well defined as $G\ast u\ast v\ast u=G\ast v\ast u\ast v$.
This can be also seen by the following lemma, which yields an equivalent definition of pivoting.
For a graph~$H$ and sets $X,Y\subseteq V(H)$, we denote by $H\ast(X,Y)$ the graph $(V(H),E(H)\triangle\{x,y:x\in X,~y\in Y,~x\neq y\})$.

\begin{lemma}[Oum~\cite{Oum2004}]\label{lem:Oum2004}
    Let~$uv$ be an edge of a graph~$G$.
    Let
    \begin{align*}
      C&:=N_G(u)\cap N_G(v),\\
      L&:=N_G(u)\setminus N_G[v],\\
      R&:=N_G(v)\setminus N_G[u].
    \end{align*}
    Then $G\wedge uv$ is identical to the graph obtained from $G\ast(C,L)\ast(L,R)\ast(R,C)$ by exchanging~$u$ and~$v$ (so that~$u$ is adjacent to $C\cup R\cup \{v\}$ and~$v$ is adjacent to $C\cup L\cup \{u\}$).
\end{lemma}

Thus, pivoting $uv$ is the operation of complementing the edge relation between 
\begin{itemize}
    \item the set of common neighbors of~$u$ and~$v$,
    \item the set of neighbors of~$u$ that are not adjacent nor equal to~$v$, and
    \item the set of neighbors of~$v$ that are not adjacent nor equal to~$u$,
\end{itemize}
and exchanging~$u$ and~$v$.
The following lemma immediately follows from \cref{lem:Oum2004}.

\begin{lemma}\label{lem:shortening}
    Let~$u$ be a vertex of a graph~$G$ having exactly two neighbors~$v$ and~$w$.
    Then\
    \[
        E(G\wedge uv-\{u,v\})=E(G-\{u,v,w\})\cup\{xw:x\in (N_G(v)\triangle N_G(w))\setminus\{v,w\}\}.
    \]
\end{lemma}

A \emph{vertex-minor} of~$G$ is a graph obtained from~$G$ by applying a sequence of vertex deletions and local complementations to~$G$.
Similarly, a \emph{pivot-minor} of~$G$ is a graph obtained from~$G$ by applying a sequence of vertex deletions and pivotings to~$G$.
Note that every pivot-minor is a vertex-minor.

It is shown that $K_t\tri K_t$ has no path on five vertices as a pivot-minor.

\begin{lemma}[Kwon, McCarty, Oum, and Wollan~\cite{KMOW2019}]
    For a positive integer~$t$, $K_t\tri K_t$ has no pivot-minor isomorphic to~$P_5$.
\end{lemma}

\subsection{Shrub-depth}

The shrub-depth is a dense analog of tree-depth, introduced by Ganian, Hlin\v{e}n\'{y}, Ne{\v{s}}et{\v{r}}il, Obdr\v{z}\'{a}lek, Ossona de Mendez, and Ramadurai~\cite{GHNOOR2012}.

Let $d\geq0$ and $m\geq1$ be integers.
A \emph{$(d,m)$-tree-model} of a graph~$G$ is a triple $(T,S,\lambda)$ of a tree~$T$ rooted at a node~$r$, a set $S\subseteq[m]^2\times[d]$, and a function $\lambda:L(T)\to[m]$ such that
\begin{itemize}
    \item for each $v\in L(T)$, $\dist_T(r,v)=d$,
    \item $V(G)$ is identical to~$L(T)$,
    \item two vertices $u$ and $v$ of $G$ are adjacent in~$G$ if and only if $(\lambda(u),\lambda(v),\frac{1}{2}\dist_T(u,v))\in S$.
\end{itemize}
We remark that any two vertices of~$G$ have even distance in~$T$ by the first condition.
Note that every graph~$G$ on~$n$ vertices admits a $(1,n)$-tree-model.
In addition, $G$ admits a $(0,m)$-tree-model for some ${m\geq0}$ if and only if ${n\leq1}$.
It is shown that if a graph~$G$ admits a $(d,m)$-tree-model, then every induced subgraph of~$G$ admits a $(d,m)$-tree-model~\cite{GHNOO2017}.

The \emph{shrub-depth} of a graph class~$\mathcal{C}$ is the smallest integer~$d$ for which there exists an integer~$m$ such that every graph in~$\mathcal{C}$ admits a $(d,m)$-tree-model.
We remark that if~$\mathcal{C}$ contains finitely many graphs, then it has shrub-depth at most~$1$ because~$m$ can be chosen as $\max_{G\in\mathcal{C}}\abs{V(G)}$.

Ganian et al.~\cite{GHNOO2017} provided the relationship between shrub-depth and the two parameters, tree-depth and linear rank-width.

\begin{proposition}[Ganian, Hlin\v{e}n\'{y}, Ne{\v{s}}et{\v{r}}il, Obdr\v{z}\'{a}lek, and Ossona de Mendez~{\cite[Proposition~3.4]{GHNOO2017}}]
    Let~$\mathcal{C}$ be a graph class and let $d\geq0$ be an integer.
    Then the following hold.
    \begin{enumerate}[label=(\alph*)]
        \item\label{item1} If~$\mathcal{C}$ is of tree-depth at most~$d$, then it is of shrub-depth at most~$d$.
        \item\label{item2} If~$\mathcal{C}$ has bounded shrub-depth, then it has bounded linear rank-width.
    \end{enumerate}
    The class of all complete graphs is a counterexample to the converse of~\ref{item1}, and the class of all paths is a counterexample to the converse of~\ref{item2}.
\end{proposition}

\subsection{Rank-depth}

The rank-depth of a graph was first introduced by DeVos, Kwon, and Oum~\cite{DKO2019}.
To define rank-depth, we use the following definitions.

For a graph~$G$ and a set $S\subseteq V(G)$, let $\rho_G(S)$ be the rank of $S\times(V(G)\setminus S)$ submatrix of the adjacency matrix of~$G$.
A \emph{decomposition} of~$G$ is a pair $(T,\sigma)$ of a tree~$T$ and a bijection~$\sigma$ from $V(G)$ to $L(T)$.
For each non-leaf node~$v$ of~$T$, let~$\mathcal{P}_v$ be a partition $\{\sigma^{-1}(V(C)\cap L(T)):\text{$C$ is a component of $T-v$}\}$ of~$V(G)$. 
The \emph{width} of~$v$ is
\[
    \max_{\mathcal{P}'\subseteq\mathcal{P}_v}\rho_G\left(\bigcup_{X\in\mathcal{P}'}X\right).
\]
The \emph{width} of a decomposition $(T,\sigma)$ of~$G$ is the maximum width of a non-leaf node of~$T$.

For integers $k,r\geq0$, a \emph{$(k,r)$-decomposition} of~$G$ is a decomposition $(T,\sigma)$ of~$G$ of width at most~$k$ such that the radius of~$T$ is at most~$r$.
Finally, the \emph{rank-depth} of~$G$ is the minimum integer~$k$ such that~$G$ admits a $(k,k)$-decomposition.
If~$G$ has at most one vertex, then its rank-depth is defined by~$0$.
Otherwise, the rank-depth of~$G$ is at least~$1$.

DeVos, Kwon, and Oum~\cite{DKO2019} showed that if $H$ is a vertex-minor of~$G$, then the rank-depth of~$H$ is less than or equal to the rank-depth of ~$G$.
They also showed that rank-depth is functionally equivalent to shrub-depth.

\begin{theorem}[DeVos, Kwon, and Oum~\cite{DKO2019}]
    A class of graphs has bounded rank-depth if and only if it has bounded shrub-depth.
\end{theorem}

\section{Flipped graphs}\label{sec:flipped graphs}

Let~$H$ be a graph and let~$\mathcal{P}$ be a collection of pairwise disjoint non-empty subsets of $V(H)$.
For each vertex~$v$ of~$H$, if~$v$ is contained in a member of~$\mathcal{P}$, then we denote the member by~$\mathcal{P}(v)$, and otherwise we let $\mathcal{P}(v):=\{v\}$.
A subset~$F$ of~$\mathcal{P}^2$ is \emph{symmetric} if for all $X,X'\in\mathcal{P}$, $(X,X')\in F$ implies $(X',X)\in F$.
For a symmetric ${F\subseteq\mathcal{P}^2}$, the \emph{$(\mathcal{P},F)$-flip} of~$H$, denoted by $H\oplus(\mathcal{P},F)$, is the graph with vertex set $V(H)$ such that distinct vertices~$u$ and~$v$ are adjacent in $H\oplus(\mathcal{P},F)$ if and only if either
\begin{itemize}
    \item $uv\notin E(H)$ and $(\mathcal{P}(u),\mathcal{P}(v))\in F$, or
    \item $uv\in E(H)$ and $(\mathcal{P}(u),\mathcal{P}(v))\notin F$.
\end{itemize}
A \emph{$\mathcal{P}$-flip} of~$H$ is a graph $H\oplus(\mathcal{P},F)$ for some symmetric $F\subseteq\mathcal{P}^2$.
If $\abs{\mathcal{P}}\leq k$, then we call it a \emph{$k$-flip} of~$H$.
For a non-empty set $X\subseteq V(H)$, the \emph{$X$-flip} of~$H$ is the graph $H\oplus(\{X\},\{(X,X)\})$.
For notational convenience, we define the $\emptyset$-flip of $H$ to be $H$ itself.
For positive integers~$m$ and~$n$, a \emph{flipped $mP_n$} is a $\mathcal{P}$-flip of $mP_n$ where~$\mathcal{P}$ is a coarsening of the column set of~$mP_n$.
If $\abs{\mathcal{P}}\leq k$, then we call it a \emph{$k$-flipped $mP_n$}.

M\"{a}hlmann~\cite{Mahlmann25} provided the list of all obstructions in terms of induced subgraphs for bounded rank-depth.

\begin{theorem}[M\"{a}hlmann~\cite{Mahlmann25}]\label{thm:Mahlmann25}
    There exists a function $g:\mathbb{N}\to\mathbb{N}$ such that for every $s\in\mathbb{N}$, every graph of rank-depth at least~$g(s)$ has an induced subgraph isomorphic to $K_s\tri K_s$, $K_s\tri\overline{K_s}$, $\overline{K_s}\tri\overline{K_s}$, or a flipped $sP_s$.
\end{theorem}

To prove \cref{thm:main}, we are going to apply \cref{thm:Mahlmann25} for an integer~$s$ which is sufficiently larger than~$t$, and then from each of the outcomes, find~$P_t$ or $K_t\tri K_t$ as a pivot-minor.
For the first outcome of \cref{thm:Mahlmann25}, we are done as every induced subgraph is a pivot-minor.
In the remaining of this paper, we will find~$P_t$ as a pivot-minor from each of the other outcomes.

The following two lemmas handle the second and third outcomes of \cref{thm:Mahlmann25}, respectively.

\begin{lemma}[Hlin\v{e}n\'{y}, Kwon, Obdr\v{z}\'{a}lek, and Ordyniak~{\cite[Proposition 6.2]{HlinenyKJS2016}}]\label{lem:second}
    For every integer~$t\geq1$, $K_t\tri\overline{K_t}$ has a pivot-minor isomorphic to~$P_{t+1}$.
\end{lemma}

\begin{lemma}[Kwon and Oum~{\cite[Lemma 2.8(1)]{KO2014}}]\label{lem:third}
    For every integer~$t\geq1$, $\overline{K_t}\tri\overline{K_t}$ has a pivot-minor isomorphic to~$P_{2t}$.
\end{lemma}

It remains to find~$P_t$ as a pivot-minor from a flipped $sP_s$.
We do this with the following two steps: we first show that every flipped $(4n-3)P_n$ has a $1$-flip of~$P_n$ as a pivot-minor, and then show that if~$n$ is large enough, then every $1$-flip of~$P_n$ has~$P_t$ as a pivot-minor.

In the remaining of this section, we will prove the following proposition for the first step.

\begin{proposition}\label{prop:preprocess}
    For every integer $n\geq1$,
    every flipped $(4n-3)P_n$ has a pivot-minor isomorphic to a $1$-flip of~$P_n$.
\end{proposition}

To prove \cref{prop:preprocess}, we will use the following notations and lemmas.
Let~$G$ be the $(\mathcal{P},F)$-flip of a graph~$H$ for a collection~$\mathcal P$ of pairwise disjoint non-empty subsets of~$V(H)$ and a symmetric ${F\subseteq\mathcal{P}^2}$.
For a subgraph~$H'$ of~$H$ and each set $X\in\mathcal{P}$, let
\begin{align*}
    X|_{H'}&:=X\cap V(H'),\\
    \mathcal{P}|_{H'}&:=\{X|_{H'}:X\in\mathcal{P},~X|_{H'}\neq\emptyset\},\\
    F|_{H'}&:=\{(X|_{H'},X'|_{H'}):(X,X')\in F,~X|_{H'}\neq\emptyset\neq X'|_{H'}\}.
\end{align*}
For convenience, if $H=mP_n$ for some $m,n\ge1$, then for each~$\ell\in [m]$, we write $X|_\ell$, $\mathcal P|_{\ell}$, and $F|_\ell$ for~$X|_{H'}$, $\mathcal{P}|_{H'}$, and~$F|_{H'}$ where $H'$ is the union of all $i$-th rows of $mP_n$ with $i\in [\ell]$, which is isomorphic to $\ell P_n$.
For a pair $(X,X')\in\mathcal{P}^2$, let
\begin{align*}
    \mathcal{C}_F(X,X')&:=\{Y\in\mathcal{P}:(X,Y)\in F,~(X',Y)\in F\},\\
    \mathcal{L}_F(X,X')&:=\{Y\in\mathcal{P}:(X,Y)\in F,~(X',Y)\notin F\},\\
    \mathcal{R}_F(X,X')&:=\{Y\in\mathcal{P}:(X,Y)\notin F,~(X',Y)\in F\}.
\end{align*}
In addition, we define $\mathcal{D}_F(X,X')$ as 
\[ 
    (\mathcal C\times\mathcal L)
    \cup (\mathcal C\times\mathcal R)
    \cup (\mathcal L\times\mathcal R)
    \cup (\mathcal L\times\mathcal C)
    \cup (\mathcal R\times\mathcal C)
    \cup (\mathcal R\times\mathcal L)
\]
where $\mathcal{C}:=\mathcal{C}_F(X,X')$, $\mathcal{L}:=\mathcal{L}_F(X,X')$, and $\mathcal{R}:=\mathcal{R}_F(X,X')$.
Note that $\mathcal{D}_F(X,X')$ is symmetric.

\begin{lemma}\label{lem:pivoting}
    Let~$H$ be a graph, let~$\mathcal P$ be a collection of pairwise disjoint non-empty subsets of $V(H)$, and let~$F$ be a symmetric subset of $\mathcal{P}^2$.
    If the $(\mathcal{P},F)$-flip~$G$ of~$H$ has an edge~$uv$ between distinct ${X,X'\in\mathcal{P}}$, then for every subgraph~$H'$ of~$H$ whose vertex set is disjoint from $N_H[\{u,v\}]$, we have
    \[
        (G\wedge uv)[V(H')]=H'\oplus(\mathcal{P}|_{H'},(F\triangle\mathcal{D}_F(X,X'))|_{H'}).
    \]
\end{lemma}
\begin{proof}
    Let $G':=(G\wedge uv)[V(H')]$.
    Without loss of generality, assume that ${u\in X}$ and ${v\in X'}$.
    Since $V(G')$ and $N_H[u]$ are disjoint, a vertex~$w\in V(G')$ is adjacent to~$u$ in~$G$ if and only if ${\mathcal{P}(w)\in\mathcal{C}_F(X,X')\cup\mathcal{L}_F(X,X')}$.
    Similarly, a vertex~$w\in V(G')$ is adjacent to~$v$ in~$G$ if and only if ${\mathcal{P}(w)\in\mathcal{C}_F(X,X')\cup\mathcal{R}_F(X,X')}$.
    Therefore, $G'=H'\oplus(\mathcal{P}|_{H'},(F\triangle\mathcal{D}_F(X,X'))|_{H'})$ by \cref{lem:Oum2004}.
\end{proof}

\begin{lemma}\label{lem:one edge}
    For integers $m,n\geq2$, let $\mathcal P$ be a coarsening of the column set of $mP_n$ and let $F$ be a symmetric subset of $\mathcal P^2$.
    If $(X,X')\in F$ for some $X,X'\in \mathcal{P}$, then the $(\mathcal{P},F)$-flip of $mP_n$ has at least one edge between~$X$ and~$X'$ whose one end is in the $(m-1)$-th row of~$mP_n$ and the other end is in the $m$-th row of~$mP_n$.
\end{lemma}
\begin{proof}
    Since~$mP_n$ has no edge between distinct rows, the $(\mathcal{P},F)$-flip of $mP_n$ has all edges between the set of vertices of $X$ in the $(m-1)$-th row and the set of vertices of $X'$ in the $m$-th row of $mP_n$.
\end{proof}

Let $\mathcal P$ be a coarsening of the column set of $(m+2)P_n$ with $\abs{\mathcal{P}}=k\geq2$ and let $F$ be a symmetric subset of $\mathcal P^2$.
Let~$G$ be the $(\mathcal{P},F)$-flip of $(m+2)P_n$.
In the following, we show that if~$\mathcal{P}$ contains 
\begin{itemize}
    \item a set~$X$ with $(X,X)\notin F$, or
    \item distinct sets~$X$ and~$X'$ with $(X,X')\notin F$ such that~$G$ has a $\bigcup\mathcal{P}$-path between~$X$ and~$X'$,
\end{itemize}
then $G$ contains a $(k-1)$-flipped $mP_n$ as a pivot-minor.

\begin{figure}
    \centering
    \begin{subfigure}{0.4\linewidth}
        \centering
        \begin{tikzpicture}[scale=0.8]
        \pgfdeclarelayer{background}
        \pgfdeclarelayer{foreground}
        \pgfsetlayers{background,main,foreground}
        \tikzstyle{v}=[circle, draw, solid, fill=black, inner sep=0pt, minimum width=3pt]
        \tikzset{c1/.style={purple, line width=6pt,opacity=0.5,line cap=round,shorten >=-2pt, shorten <=-2pt}}
        
        \foreach \i in {1,2,5,6,7,8}{
            \node [v] (v\i) at (0,5-\i*.5) {};
            \node [v] (w\i) at (1,5-\i*.5) {};
            \node [v] (z\i) at (2,5-\i*.5) {};
            \node [v] (a\i) at (3,5-\i*.5) {};
            \node [v] (b\i) at (4,5-\i*.5) {};
            \node [v] (c\i) at (5,5-\i*.5) {};
            \node [v] (d\i) at (6,5-\i*.5) {};
            \node [v] (f\i) at (7,5-\i*.5) {};
            \draw (v\i)--(w\i)--(z\i)--(a\i);
            \draw (b\i)--(c\i)--(d\i)--(f\i);
        }
        \foreach \i in {1,2,5,6,7,8}{
            \foreach \j in {1,2,5,6,7,8}{
		      \ifthenelse{\j=\i}{}{\draw(a\i)--(b\j);}
	        }
        }
        \foreach \i in {0,3,6}{
        \draw[rounded corners, thick] (0.5+\i,5)--(-0.2+\i,5)--(-0.2+\i,0.5)--(1.2+\i,0.5)--(1.2+\i,5)--(0.4+\i,5);
        }
        \node [label=$X$] (w) at (0.5,5){};
        \node [label=$X_1$] (w) at (3.5,5){};
        \node [label=$X_2$] (w) at (6.5,5){};
        \node [label=$\vdots$] (w) at (2,2.8){};
        \node [label=$\vdots$] (w) at (5,2.8){};
        
        \draw[very thick]    (0.5, 0.5) to[out=-30,in=-150] (3.5,0.5);
        \draw[very thick]    (3.5, 0.5) to[out=-30,in=-150] (6.5,0.5);
        \draw[very thick]    (0.5, 0.5) to[out=-40,in=-140] (6.5,0.5);
        
        \draw[thick] (a1)--(a8);\draw[thick] (b1)--(b8);
        
        \begin{pgfonlayer}{background}
            \draw[c1] (w7)--(a8);
        \end{pgfonlayer}
        \draw (w7)--(a8);
        \node [label=$u$] (v) at (0.8,0.8) {};
        \node [label=$v$] (v) at (3.2,0.3) {};
    \end{tikzpicture}
    \caption{A graph $G$ which is a flipped $(m+2)P_n$.}
    \end{subfigure}
    \hspace{1cm}
    \begin{subfigure}{0.4\linewidth}
        \centering
        \begin{tikzpicture}[scale=0.8]
        \pgfdeclarelayer{background}
        \pgfdeclarelayer{foreground}
        \pgfsetlayers{background,main,foreground}
        \tikzstyle{v}=[circle, draw, solid, fill=black, inner sep=0pt, minimum width=3pt]
        \tikzset{c1/.style={purple, line width=6pt,opacity=0.5,line cap=round,shorten >=-2pt, shorten <=-2pt}}
        
        \foreach \i in {1,2,5,6}{
            \node [v] (v\i) at (0,5-\i*.5) {};
            \node [v] (w\i) at (1,5-\i*.5) {};
            \node [v] (z\i) at (2,5-\i*.5) {};
            \node [v] (a\i) at (3,5-\i*.5) {};
            \node [v] (b\i) at (4,5-\i*.5) {};
            \node [v] (c\i) at (5,5-\i*.5) {};
            \node [v] (d\i) at (6,5-\i*.5) {};
            \node [v] (f\i) at (7,5-\i*.5) {};
            \draw (v\i)--(w\i)--(z\i)--(a\i);
            \draw (b\i)--(c\i)--(d\i)--(f\i);
        }
        \foreach \i in {1,2,5,6}{
            \foreach \j in {1,2,5,6}{
		      \ifthenelse{\j=\i}{}{\draw(a\i)--(b\j);}
	        }
        }
        \foreach \i in {0,3,6}{
        \draw[rounded corners, thick] (0.5+\i,5)--(-0.2+\i,5)--(-0.2+\i,0.5)--(1.2+\i,0.5)--(1.2+\i,5)--(0.4+\i,5);
        }
        \node [label=$X$] (w) at (0.5,5){};
        \node [label=$X_1$] (w) at (3.5,5){};
        \node [label=$X_2$] (w) at (6.5,5){};
        \node [label=$\vdots$] (w) at (2,2.8){};
        \node [label=$\vdots$] (w) at (5,2.8){};
        
        \draw[very thick, dashed]   (0.5, 0.5) to[out=-30,in=-150] (3.5,0.5);
        \draw[very thick]           (3.5, 0.5) to[out=-30,in=-150] (6.5,0.5);
        \draw[very thick, dashed]   (0.5, 0.5) to[out=-40,in=-140] (6.5,0.5);
        
        \draw[thick] (a1)--(a6);\draw[thick] (b1)--(b6);
    \end{tikzpicture}
    \caption{$G'=(G\wedge uv)\bigl[[m]\times [n]\bigr]$.}
    \end{subfigure}

    \caption{An example of the case $(X,X)\notin F$ in \cref{lem:case1} where $F$ consists of $(X, X_1)$, $(X, X_2)$, $(X_1, X_2)$, $(X_1, X_1)$ and their reverses.
    Note that $\mathcal{C}_F(X, X_1)=\{X_1, X_2\}$ and $\mathcal{R}_F(X, X_1)=\{X\}$.
    One can observe that $F\triangle\mathcal{D}_F(X,X_1)$ contains no pair $(X,X')$ for any $X'\in\mathcal{P}$, and $G'=(mP_n)\oplus(\mathcal{P}|_m,(F\triangle\mathcal{D}_F(X,X_1))|_m)$.}
    \label{fig:notXX}
\end{figure}
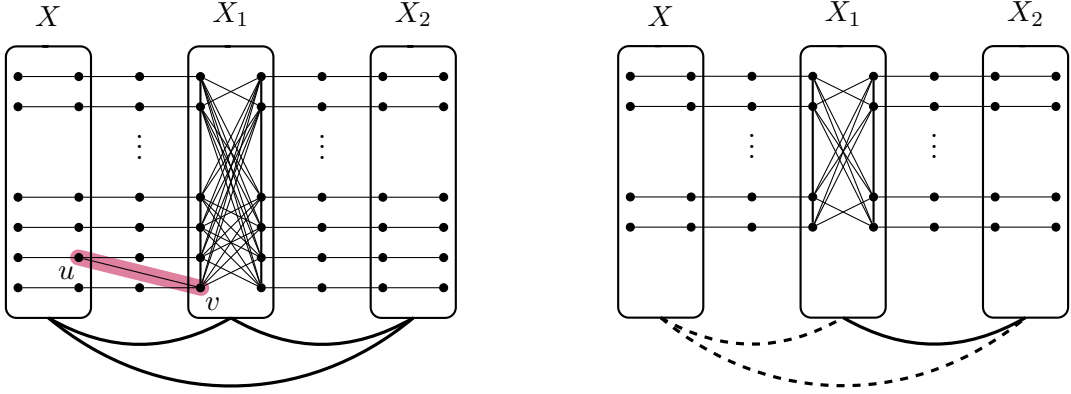

\begin{lemma}\label{lem:case1}
    For integers $m,n\geq1$, let $\mathcal P$ be a coarsening of the column set of $(m+2)P_n$ with $\abs{\mathcal{P}}=k\geq2$ and let~$F$ be a symmetric subset of $\mathcal P^2$.
    If~$\mathcal{P}$ contains a set~$X$ with ${(X,X)\notin F}$, then the $(\mathcal{P},F)$-flip of $(m+2)P_n$ contains a $(k-1)$-flipped $mP_n$ as a pivot-minor.
\end{lemma}
\begin{proof}
    Let~$G$ be the $(\mathcal{P},F)$-flip of $(m+2)P_n$ and let~$X\in \mathcal{P}$ with $(X,X)\notin F$.
    Let~$d$ be the number of sets $X'\in\mathcal{P}$ with $(X,X')\in F$.
    If $d=0$, then~$G$ is a $(\mathcal{P}\setminus\{X\},F)$-flip of $(m+2)P_n$, which is a $(k-1)$-flipped $(m+2)P_n$.
    Thus, we may assume that $d>0$.
    Let $X_1,\ldots,X_d$ be the sets in~$\mathcal{P}$ such that for each $i\in[d]$, $(X,X_i)\in F$.
    Note that $X\neq X_i$ for each $i\in[d]$ as ${(X,X)\notin F}$.
    
    By \cref{lem:one edge}, $G$ has an edge~$uv$ with $u\in X$ and $v\in X_1$ such that~$u$ and~$v$ are in the $(m+1)$-th and $(m+2)$-th rows of $(m+2)P_n$, respectively.
    Since~$(m+2)P_n$ has no edge between distinct rows, $N_{(m+2)P_n}[\{u,v\}]$ is disjoint from $mP_n$.
    Thus, by \cref{lem:pivoting},
    \[
        G':=(G\wedge uv)\bigl[[m]\times[n]\bigr]=(mP_n)\oplus(\mathcal{P}|_m,(F\triangle\mathcal{D}_F(X,X_1))|_m).
    \]
    See \cref{fig:notXX} for an illustration.
    
    We show that~$G'$ is a $(k-1)$-flipped $mP_n$.
    Since the two neighbors of~$u$ in $(m+2)P_n$ are on the $(m+1)$-th row of $(m+2)P_n$, no vertex of $V(G')$ is adjacent to~$u$ in~$(m+2)P_n$.
    Similarly, no vertex of $V(G')$ is adjacent to~$v$ in~$(m+2)P_n$.
    Therefore, a vertex $w\in V(G')$ is adjacent to~$u$ in~$G$ if and only if ${\mathcal{P}(w)\in\mathcal{C}_F(X,X_1)\cup\mathcal{L}_F(X,X_1)}$.
    Similarly, a vertex $w\in V(G')$ is adjacent to~$v$ in~$G$ if and only if ${\mathcal{P}(w)\in\mathcal{C}_F(X,X_1)\cup\mathcal{R}_F(X,X_1)}$.
    Since ${(X,X)\notin F}$ and ${(X,X_i)\in F}$ for each $i\in[d]$, we have that ${X\in\mathcal{R}_F(X,X_1)}$ and
    \[
        \mathcal{C}_F(X,X_1)\cup\mathcal{L}_F(X,X_1)=\{X_1,\ldots,X_d\}.
    \]
    Thus, $F\triangle\mathcal{D}_F(X,X_1)$ contains no pair $(X,X')$ for any $X'\in\mathcal{P}$.
    Hence, there is no $X'\in\mathcal{P}$ such that
    \[
        (X|_m,X'|_m)\in(F\triangle\mathcal{D}_F(X,X_1))|_m,
    \]
    so~$G'$ is a $(\mathcal{P}\setminus\{X\})|_m$-flip of $mP_n$, which is a $(k-1)$-flipped $mP_n$.
    Clearly, $G'$ is a pivot-minor of~$G$.
\end{proof}

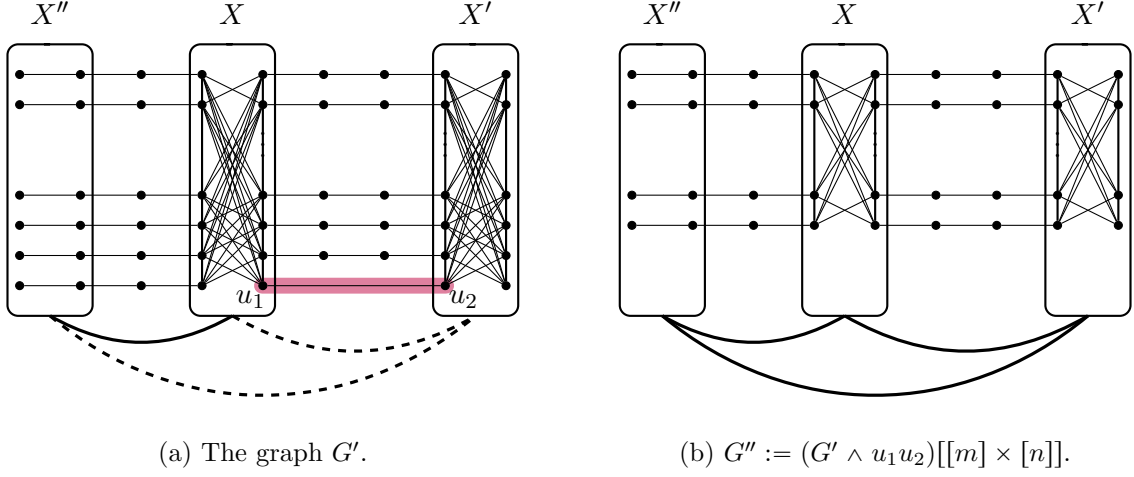
\begin{figure}
    \centering
    \begin{subfigure}{0.4\linewidth}
        \centering
        \begin{tikzpicture}[scale=0.8]
        \pgfdeclarelayer{background}
        \pgfdeclarelayer{foreground}
        \pgfsetlayers{background,main,foreground}
        \tikzstyle{v}=[circle, draw, solid, fill=black, inner sep=0pt, minimum width=3pt]
        \tikzset{c1/.style={purple, line width=6pt,opacity=0.5,line cap=round,shorten >=-2pt, shorten <=-2pt}}
        
            \foreach \i in {8}{
            \node [v] (g\i) at (-2,5-\i*.5) {};
            \node [v] (h\i) at (-1,5-\i*.5) {};
            \node [v] (v\i) at (0,5-\i*.5) {};
            \node [v] (w\i) at (1,5-\i*.5) {};
            \node [v] (z\i) at (2,5-\i*.5) {};
            \node [v] (c\i) at (5,5-\i*.5) {};
            \node [v] (d\i) at (6,5-\i*.5) {};
            \draw (g\i)--(h\i)--(v\i)--(w\i);
            \draw (z\i)--(c\i);
        }

        \foreach \i in {1,2,5,6,7}{
            \node [v] (g\i) at (-2,5-\i*.5) {};
            \node [v] (h\i) at (-1,5-\i*.5) {};
            \node [v] (v\i) at (0,5-\i*.5) {};
            \node [v] (w\i) at (1,5-\i*.5) {};
            \node [v] (z\i) at (2,5-\i*.5) {};
            \node [v] (a\i) at (3,5-\i*.5) {};
            \node [v] (b\i) at (4,5-\i*.5) {};
            \node [v] (c\i) at (5,5-\i*.5) {};
            \node [v] (d\i) at (6,5-\i*.5) {};
            \draw (g\i)--(h\i)--(v\i)--(w\i);
            \draw (z\i)--(a\i)--(b\i)--(c\i);
        }
        \foreach \i in {1,2,5,6,7,8}{
            \foreach \j in {1,2,5,6,7,8}{
		      \ifthenelse{\j=\i}{}{\draw(w\i)--(z\j);}
		      \ifthenelse{\j=\i}{}{\draw(c\i)--(d\j);}
	        }
        }
        \draw[thick] (w1)--(w8);\draw[thick] (z1)--(z8);
        \draw[thick] (c1)--(c8);\draw[thick] (d1)--(d8);

        \begin{pgfonlayer}{background}
            \draw[c1] (z8)--(c8);
        \end{pgfonlayer}
        
        \foreach \i in {-2,1,5}{
        \draw[rounded corners, thick] (0.5+\i,5)--(-0.2+\i,5)--(-0.2+\i,0.5)--(1.2+\i,0.5)--(1.2+\i,5)--(0.4+\i,5);
        }
        \node [label=$X''$] (w) at (-1.5,5){};
        \node [label=$X$] (w) at (1.5,5){};
        \node [label=$X'$] (w) at (5.5,5){};
        \node [label=$\vdots$] (w) at (2,2.8){};
        \node [label=$\vdots$] (w) at (5,2.8){};
        
        \draw[very thick, dashed]   (1.5, 0.5) to[out=-30,in=-150] (5.5,0.5);
        \draw[very thick]           (-1.5, 0.5) to[out=-30,in=-150] (1.5,0.5);
        \draw[very thick, dashed]   (-1.5, 0.5) to[out=-40,in=-140] (5.5,0.5);
        
        \node [label=$u_1$] (v) at (1.8,0.3) {};
        \node [label=$u_2$] (v) at (5.3,0.3) {};
    \end{tikzpicture}
    \caption{The graph $G'$.}
    \end{subfigure}
    \hspace{1cm}
    \begin{subfigure}{0.4\linewidth}
        \centering
        \begin{tikzpicture}[scale=0.8]
        \pgfdeclarelayer{background}
        \pgfdeclarelayer{foreground}
        \pgfsetlayers{background,main,foreground}
        \tikzstyle{v}=[circle, draw, solid, fill=black, inner sep=0pt, minimum width=3pt]
        \tikzset{c1/.style={purple, line width=6pt,opacity=0.5,line cap=round,shorten >=-2pt, shorten <=-2pt}}
        
        \foreach \i in {1,2,5,6}{
            \node [v] (g\i) at (-2,5-\i*.5) {};
            \node [v] (h\i) at (-1,5-\i*.5) {};
            \node [v] (v\i) at (0,5-\i*.5) {};
            \node [v] (w\i) at (1,5-\i*.5) {};
            \node [v] (z\i) at (2,5-\i*.5) {};
            \node [v] (a\i) at (3,5-\i*.5) {};
            \node [v] (b\i) at (4,5-\i*.5) {};
            \node [v] (c\i) at (5,5-\i*.5) {};
            \node [v] (d\i) at (6,5-\i*.5) {};
            \draw (g\i)--(h\i)--(v\i)--(w\i);
            \draw (z\i)--(a\i)--(b\i)--(c\i);
        }
        \foreach \i in {1,2,5,6}{
            \foreach \j in {1,2,5,6}{
		      \ifthenelse{\j=\i}{}{\draw(w\i)--(z\j);}
		      \ifthenelse{\j=\i}{}{\draw(c\i)--(d\j);}
	        }
        }
        \draw[thick] (w1)--(w6);\draw[thick] (z1)--(z6);
        \draw[thick] (c1)--(c6);\draw[thick] (d1)--(d6);
        
        \foreach \i in {-2,1,5}{
        \draw[rounded corners, thick] (0.5+\i,5)--(-0.2+\i,5)--(-0.2+\i,0.5)--(1.2+\i,0.5)--(1.2+\i,5)--(0.4+\i,5);
        }
        \node [label=$X''$] (w) at (-1.5,5){};
        \node [label=$X$] (w) at (1.5,5){};
        \node [label=$X'$] (w) at (5.5,5){};
        \node [label=$\vdots$] (w) at (2,2.8){};
        \node [label=$\vdots$] (w) at (5,2.8){};
        
        \draw[very thick]    (1.5, 0.5) to[out=-30,in=-150] (5.5,0.5);
        \draw[very thick]    (-1.5, 0.5) to[out=-30,in=-150] (1.5,0.5);
        \draw[very thick]    (-1.5, 0.5) to[out=-40,in=-140] (5.5,0.5);
    \end{tikzpicture}
    \caption{$G'':=(G'\wedge u_1u_2)[[m]\times [n]]$.}
    \end{subfigure}
    
    \caption{An example of~$G'$ in the induction statement in \cref{lem:case2} when $\ell=2$.
    Observe that $X''\in\mathcal{L}_F(X, X')$, and both $(X'',X)$ and $(X'',X')$ are contained in ${F\triangle\mathcal D_{F}(X,X')}$.
    Also, $(X,X')\in {F\triangle\mathcal D_{F}(X,X')}$.
    In this case, we obtain a coarsening of~$\mathcal{P}$ by merging~$X$ and~$X'$.}
    \label{fig:mergingXbase1}
\end{figure}

\begin{lemma}\label{lem:case2}
    For integers $m,n\geq1$, let $\mathcal P$ be a coarsening of the column set of $(m+2)P_n$ with $\abs{\mathcal{P}}=k\geq2$ and let~$F$ be a symmetric subset of $\mathcal P^2$.
    Let~$X$ and~$X'$ be distinct sets in~$\mathcal{P}$ such that $(X,X),(X',X')\in F$ and $(X,X')\notin F$.
    If the $(\mathcal{P},F)$-flip~$G$ of $(m+2)P_n$ has a $\bigcup\mathcal{P}$-path~$Q$ from~$X$ to~$X'$, then $G$ contains a $(k-1)$-flipped $mP_n$ as a pivot-minor.
\end{lemma}
\begin{proof}
    Since~$Q$ is a $\bigcup\mathcal{P}$-path in~$G$ from $X$ to $X'$ and $(X,X')\notin F$, it is a subpath of~$(m+2)P_n$.
    We may assume that~$Q$ is a subpath of the $(m+2)$-th row of $(m+2)P_n$.
    Note that every internal vertex of~$Q$ has degree~$2$ in~$G$.
    Let $q:=\abs{V(Q)}$.
    As~$X$ and~$X'$ are distinct, $q$ is at least~$2$.
    
    We are going to prove a stronger statement: 
    \begin{quote}
        For every integer~$\ell$ with $2\leq\ell\leq q$ and $\ell\equiv q\pmod{2}$, if a graph~$G'$ is obtained from~$G$ by replacing~$Q$ with a path $u_1u_2\cdots u_\ell$, where ${u_1\in X}$ and ${u_\ell\in X'}$ are the ends of $Q$ and $u_2,\dots,u_{\ell-1}\notin V(G)$, then~$G'$ contains a $(k-1)$-flipped $mP_n$ as a pivot-minor.
    \end{quote}
    Let $H$ be the graph obtained from $(m+2)P_n$ by replacing $Q$ with the same path $u_1u_2\cdots u_\ell$.
    Clearly, $G'$ is the $(\mathcal{P},F)$-flip of $H$.
    
    We proceed by induction on~$\ell$.
    First, let us assume that $\ell=2$.
    See~\cref{fig:mergingXbase1} for an illustration.
    In this case, $u_1u_2$ is an edge between~$X$ and~$X'$.
    Let
    \begin{align*}
        G''&:=(G'\wedge u_1u_2)\bigl[[m]\times[n]\bigr],\\
        F''&:=(F\triangle\mathcal{D}_F(X,X'))|_m.
    \end{align*}
    Since~$(m+2)P_n$ has no edge between distinct rows and $u_1, u_2$ are contained in the $(m+2)$-th row, $N_{H}[\{u_1,u_2\}]$ is disjoint from $mP_n$ in $H$.
    Thus by \cref{lem:pivoting}, $G''=(mP_n)\oplus(\mathcal{P}|_m,F'')$.
    Also, note that~$G''$ is a pivot-minor of~$G'$.

    We claim that~$G''$ is a $(k-1)$-flipped $mP_n$.
    Since $u_1$ has no neighbors in~$V(G'')$ in~$(m+2)P_n$, a vertex $w\in V(G'')$ is adjacent to~$u_1$ in~$G'$ if and only if $\mathcal{P}(w)\in\mathcal{C}_F(X,X')\cup\mathcal{L}_F(X,X')$.
    Similarly, a vertex $w\in V(G'')$ is adjacent to~$u_2$ in~$G'$ if and only if $\mathcal{P}(w)\in\mathcal{C}_F(X,X')\cup\mathcal{R}_F(X,X')$.
    Since $(X,X),(X',X')\in F$ and $(X,X')\notin F$, we have that $X\in\mathcal{L}_F(X,X')$ and $X'\in\mathcal{R}_F(X,X')$.

    Observe that for every $Z\in\mathcal{P}\setminus \{X,X'\}$,
    \begin{itemize}
        \item if $Z\in \mathcal{L}_F(X,X')$, then $(X,Z)\notin \mathcal{D}_F(X,X')$, $(X',Z)\in \mathcal{D}_F(X,X')$, and $(X,Z),(X',Z)\in F\triangle\mathcal{D}_F(X,X')$,
        \item if $Z\in \mathcal{R}_F(X,X')$, then $(X,Z)\in \mathcal{D}_F(X,X')$, $(X',Z)\notin \mathcal{D}_F(X,X')$, and $(X,Z),(X',Z)\in F\triangle\mathcal{D}_F(X,X')$, 
        \item if $Z\in \mathcal{C}_F(X,X')$,
        then $(X,Z),(X',Z)\in \mathcal{D}_F(X,X')$ and $(X,Z), (X',Z)\notin F\triangle\mathcal{D}_F(X,X')$.
    \end{itemize}
    Also, $(X,X')\in F\triangle \mathcal{D}_F(X,X')$.
    Therefore, for every ${Z\in\mathcal{P}}$, ${(X,Z)\in F\triangle\mathcal D_{F}(X,X')}$ if and only if ${(X',Z)\in F\triangle\mathcal D_{F}(X,X')}$.
    
    Let ${\mathcal{P}_0:=(\mathcal{P}\setminus\{X,X'\})\cup\{X\cup X'\}}$.
    Let~$F_0$ be the set obtained from~$F$ by removing all pairs which contain~$X$ or~$X'$ and adding the pairs in
    \[
        \{X\cup X'\}\times\Bigl(\bigl(\{X\cup X'\}\cup\mathcal{L}_F(X,X')\cup\mathcal{R}_F(X,X')\bigr)\setminus\{X,X'\}\Bigr)
    \]
    and their reverses.
    Note that~$F_0$ is a symmetric subset of~$\mathcal{P}_0^2$.
    By the observation, 
    the ${(\mathcal P,F\triangle\mathcal{D}_F(X,X'))}$-flip is identical to the $(\mathcal P_0,F_0)$-flip, and thus~$G''$ is the $(\mathcal{P}_0|_m,F_0|_m)$-flip of~$mP_n$.
    Therefore, $G''$ is a $(k-1)$-flipped $mP_n$ and a pivot-minor of~$G'$.

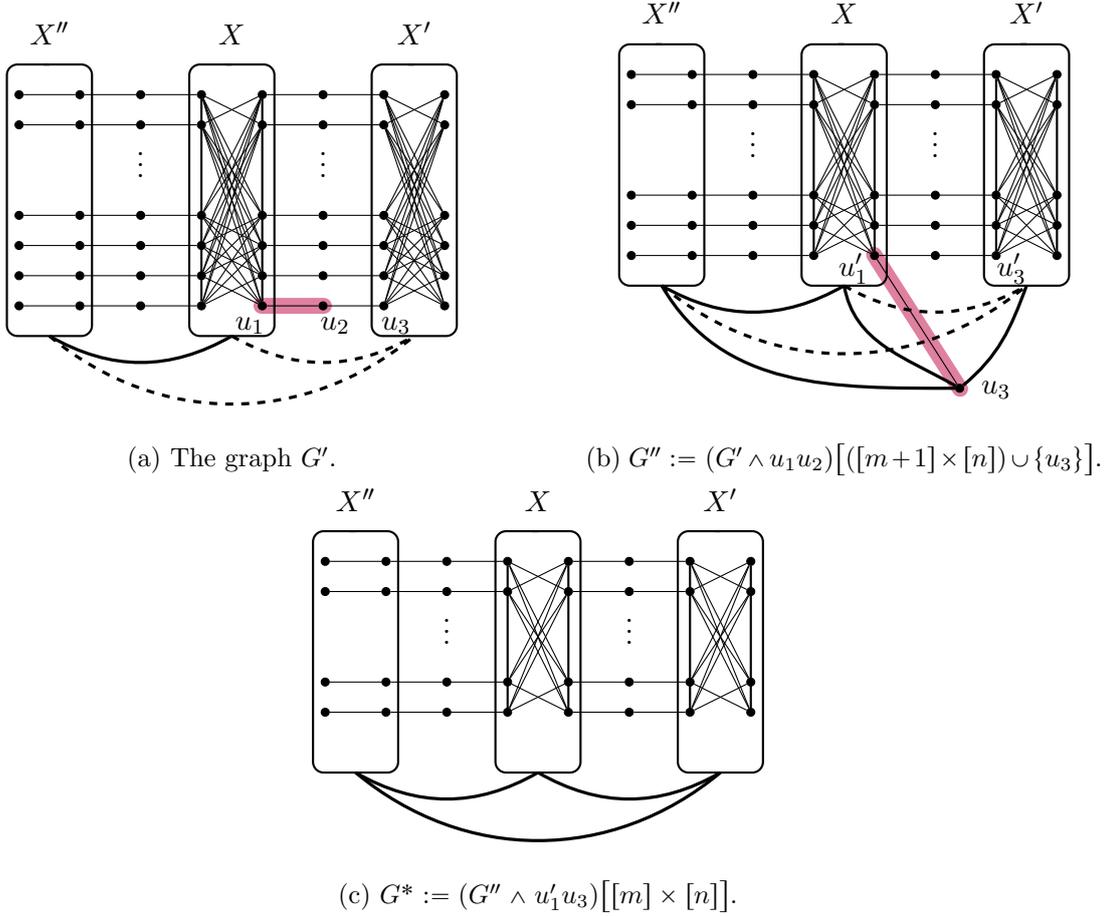
\begin{figure}
    \centering
    \begin{subfigure}{0.4\linewidth}
        \centering
        \begin{tikzpicture}[scale=0.8]
        \pgfdeclarelayer{background}
        \pgfdeclarelayer{foreground}
        \pgfsetlayers{background,main,foreground}
        \tikzstyle{v}=[circle, draw, solid, fill=black, inner sep=0pt, minimum width=3pt]
        \tikzset{c1/.style={purple, line width=6pt,opacity=0.5,line cap=round,shorten >=-2pt, shorten <=-2pt}}
        
        \foreach \i in {1,2,5,6,7,8}{
            \node [v] (v\i) at (0,5-\i*.5) {};
            \node [v] (w\i) at (1,5-\i*.5) {};
            \node [v] (z\i) at (2,5-\i*.5) {};
            \node [v] (a\i) at (3,5-\i*.5) {};
            \node [v] (b\i) at (4,5-\i*.5) {};
            \node [v] (c\i) at (5,5-\i*.5) {};
            \node [v] (d\i) at (6,5-\i*.5) {};
            \node [v] (f\i) at (7,5-\i*.5) {};
            \draw (v\i)--(w\i)--(z\i)--(a\i);
            \draw (b\i)--(c\i)--(d\i);
        }
        \foreach \i in {1,2,5,6,7,8}{
            \foreach \j in {1,2,5,6,7,8}{
		      \ifthenelse{\j=\i}{}{\draw(a\i)--(b\j);}
              \ifthenelse{\j=\i}{}{\draw(d\i)--(f\j);}
	        }
        }
        \foreach \i in {0,3,6}{
        \draw[rounded corners, thick] (0.5+\i,5)--(-0.2+\i,5)--(-0.2+\i,0.5)--(1.2+\i,0.5)--(1.2+\i,5)--(0.4+\i,5);
        }
        \node [label=$X''$] (w) at (0.5,5){};
        \node [label=$X$] (w) at (3.5,5){};
        \node [label=$X'$] (w) at (6.5,5){};
        \node [label=$\vdots$] (w) at (2,2.8){};
        \node [label=$\vdots$] (w) at (5,2.8){};
        
        \draw[very thick]    (0.5, 0.5) to[out=-30,in=-150] (3.5,0.5);
        \draw[very thick,dashed]    (3.5, 0.5) to[out=-30,in=-150] (6.5,0.5);
        \draw[very thick,dashed]    (0.5, 0.5) to[out=-40,in=-140] (6.5,0.5);
        
        \draw[thick] (a1)--(a8);\draw[thick] (b1)--(b8);
        
        \begin{pgfonlayer}{background}
            \draw[c1] (b8)--(c8);
        \end{pgfonlayer}
        \node [label=$u_1$] (v) at (3.8,0.2) {};
        \node [label=$u_2$] (v) at (5.2,0.2) {};
        \node [label=$u_3$] (v) at (6.2,0.2) {};
    \end{tikzpicture}
    \caption{The graph $G'$.}
    \end{subfigure}
    \hspace{1cm}
    \begin{subfigure}{0.4\linewidth}
        \centering
        \begin{tikzpicture}[scale=0.8]
        \pgfdeclarelayer{background}
        \pgfdeclarelayer{foreground}
        \pgfsetlayers{background,main,foreground}
        \tikzstyle{v}=[circle, draw, solid, fill=black, inner sep=0pt, minimum width=3pt]
        \tikzset{c1/.style={purple, line width=6pt,opacity=0.5,line cap=round,shorten >=-2pt, shorten <=-2pt}}
        
        \foreach \i in {1,2,5,6,7}{
            \node [v] (v\i) at (0,5-\i*.5) {};
            \node [v] (w\i) at (1,5-\i*.5) {};
            \node [v] (z\i) at (2,5-\i*.5) {};
            \node [v] (a\i) at (3,5-\i*.5) {};
            \node [v] (b\i) at (4,5-\i*.5) {};
            \node [v] (c\i) at (5,5-\i*.5) {};
            \node [v] (d\i) at (6,5-\i*.5) {};
            \node [v] (f\i) at (7,5-\i*.5) {};
            \draw (v\i)--(w\i)--(z\i)--(a\i);
            \draw (b\i)--(c\i)--(d\i);
        }

        \node [v] (z) at (5.4,-0.7) {};
        \node [label=$u_3$] (v) at (6,-1.2) {};
        \draw (z) -- (b7);
   
        \foreach \i in {1,2,5,6,7}{
            \foreach \j in {1,2,5,6,7}{
		      \ifthenelse{\j=\i}{}{\draw(a\i)--(b\j);}
              \ifthenelse{\j=\i}{}{\draw(d\i)--(f\j);}
	        }
        }
        \foreach \i in {0,3,6}{
        \draw[rounded corners, thick] (0.5+\i,5)--(-0.2+\i,5)--(-0.2+\i,1)--(1.2+\i,1)--(1.2+\i,5)--(0.4+\i,5);
        }
        \node [label=$X''$] (w) at (0.5,5){};
        \node [label=$X$] (w) at (3.5,5){};
        \node [label=$X'$] (w) at (6.5,5){};
        \node [label=$\vdots$] (w) at (2,2.8){};
        \node [label=$\vdots$] (w) at (5,2.8){};
        
        \draw[very thick]   (0.5, 1) to[out=-30,in=-150] (3.5,1);
        \draw[very thick, dashed]           (3.5, 1) to[out=-30,in=-150] (6.5,1);
        \draw[very thick, dashed]   (0.5, 1) to[out=-40,in=-140] (6.5,1);

        \draw[very thick] (z) to[out=180,in=-60] (0.5,1);
        \draw[very thick] (z) to[out=150,in=-80] (3.5,1);
        \draw[very thick] (z) to[out=40,in=-110] (6.5,1);
        
        \node [label=$u_1'$] (v) at (3.65,0.7) {};
        \node [label=$u_3'$] (v) at (6.25,0.7) {};

        \draw[thick] (a1)--(a7);\draw[thick] (b1)--(b7);\draw[thick] (d1)--(d7);\draw[thick] (f1)--(f7);
        \begin{pgfonlayer}{background}
            \draw[c1] (b7)--(z);
        \end{pgfonlayer}
 
    \end{tikzpicture}
    \caption{$G'':=(G'\wedge u_1u_2)\bigl[([m+1]\times[n])\cup\{u_3\}\bigr]$.}
    \end{subfigure}
    \hspace{1cm}
    \begin{subfigure}{0.4\linewidth}
        \centering
        \begin{tikzpicture}[scale=0.8]
        \pgfdeclarelayer{background}
        \pgfdeclarelayer{foreground}
        \pgfsetlayers{background,main,foreground}
        \tikzstyle{v}=[circle, draw, solid, fill=black, inner sep=0pt, minimum width=3pt]
        \tikzset{c1/.style={purple, line width=6pt,opacity=0.5,line cap=round,shorten >=-2pt, shorten <=-2pt}}
        
        \foreach \i in {1,2,5,6}{
            \node [v] (v\i) at (0,5-\i*.5) {};
            \node [v] (w\i) at (1,5-\i*.5) {};
            \node [v] (z\i) at (2,5-\i*.5) {};
            \node [v] (a\i) at (3,5-\i*.5) {};
            \node [v] (b\i) at (4,5-\i*.5) {};
            \node [v] (c\i) at (5,5-\i*.5) {};
            \node [v] (d\i) at (6,5-\i*.5) {};
            \node [v] (f\i) at (7,5-\i*.5) {};
            \draw (v\i)--(w\i)--(z\i)--(a\i);
            \draw (b\i)--(c\i)--(d\i);
        }
   
        \foreach \i in {1,2,5,6}{
            \foreach \j in {1,2,5,6}{
		      \ifthenelse{\j=\i}{}{\draw(a\i)--(b\j);}
              \ifthenelse{\j=\i}{}{\draw(d\i)--(f\j);}
	        }
        }
        \foreach \i in {0,3,6}{
        \draw[rounded corners, thick] (0.5+\i,5)--(-0.2+\i,5)--(-0.2+\i,1)--(1.2+\i,1)--(1.2+\i,5)--(0.4+\i,5);
        }
        \node [label=$X''$] (w) at (0.5,5){};
        \node [label=$X$] (w) at (3.5,5){};
        \node [label=$X'$] (w) at (6.5,5){};
        \node [label=$\vdots$] (w) at (2,2.8){};
        \node [label=$\vdots$] (w) at (5,2.8){};
        
        \draw[very thick]   (0.5, 1) to[out=-30,in=-150] (3.5,1);
        \draw[very thick]           (3.5, 1) to[out=-30,in=-150] (6.5,1);
        \draw[very thick]   (0.5, 1) to[out=-40,in=-140] (6.5,1);
       
        \draw[thick] (a1)--(a6);\draw[thick] (b1)--(b6);\draw[thick] (d1)--(d6);\draw[thick] (f1)--(f6);
 
    \end{tikzpicture}
    \caption{$G^*:=(G''\wedge u'_1u_3)\bigl[[m]\times[n]\bigr]$.}
    \end{subfigure}

    \caption{An example of~$G'$ in the induction statement in \cref{lem:case2} when $\ell=3$.
    In this case, we first pivot $u_1u_2$ and then pivot $u_1'u_3$ to obtain $G^*$.}
    \label{fig:mergingXbase2}
\end{figure}
    
    Next, let us consider the case that $\ell=3$.
    See~\cref{fig:mergingXbase2} for an illustration.
    Let
    \begin{align*}
        G''&:=(G'\wedge u_1u_2)\bigl[([m+1]\times[n])\cup\{u_3\}\bigr],\\
        H''&:=((m+2)P_n)\bigl[([m+1]\times[n])\cup\{u_3\}\bigr].
    \end{align*}
    Let~$\mathcal{P}''$ be the set obtained from~$\mathcal{P}|_{m+1}$ by adding~$\{u_3\}$ and let~$F''$ be the set obtained from~$F|_{m+1}$ by adding $(\{u_3\},Z|_{m+1})$ for every $Z\in\mathcal{L}_F(X,X')\cup\mathcal{R}_F(X,X')$ and their reverses.
    Since~$u_1$ and~$u_3$ are the only neighbors of~$u_2$ in~$G'$, by \cref{lem:shortening},
    \begin{align}
        N_{G''}(u_3)&=(N_{G'}(u_1)\triangle N_{G'}(u_3))\cap([m+1]\times[n]),\label{eq1}\\
        G''-u_3&=G'\bigl[[m+1]\times[n]\bigr]=G\bigl[[m+1]\times[n]\bigr].\notag
    \end{align}
    Therefore, $G''$ is the $(\mathcal{P}'',F'')$-flip of~$H''$, and $G''-u_3$ is the $(\mathcal{P}|_{m+1},F|_{m+1})$-flip of $(m+1)P_n$.
    Note that~$u_3$ is adjacent to each vertex in $X|_{m+1}$ in~$G''$ as $X\in\mathcal{L}_F(X,X')$.

    Let~$u'_1$ be the vertex on the $(m+1)$-th row of $(m+2)P_n$ which is in the same column as~$u_1$.
    Observe that $u'_1u_3$ is an edge of~$G''$.
    Let
    \begin{align*}
        G^*&:=(G''\wedge u'_1u_3)\bigl[[m]\times[n]\bigr],\\
        F^*&:=(F\triangle\mathcal{D}_F(X,X'))|_m.
    \end{align*}
    Note that~$G^*$ is a pivot-minor of~$G'$.
    By \eqref{eq1}, we observe that
    \begin{align*}
        \mathcal{C}_F(X,X')|_{m+1}&=\mathcal{L}_{F''}(X|_{m+1},\{u_3\}),\\
        \mathcal{L}_F(X,X')|_{m+1}&=\mathcal{C}_{F''}(X|_{m+1},\{u_3\}),\\
        \mathcal{R}_F(X,X')|_{m+1}&=\mathcal{R}_{F''}(X|_{m+1},\{u_3\}).
    \end{align*}
    Therefore, $F^*=(F''\triangle\mathcal{D}_{F''}(X,\{u_3\}))|_m$.
    By applying~\cref{lem:pivoting} for $\{u_1', u_3\}$ in $H''$, we have that $G^*$ is the $(\mathcal{P}''|_m, F^*)$-flip of $mP_n$, which is the $(\mathcal{P}|_m,F^*)$-flip of $mP_n$.

    We claim that~$G^*$ is a $(k-1)$-flipped $mP_n$.
    Similar to~$u'_1$, let~$u'_3$ be the vertex on the $(m+1)$-th row of $(m+2)P_n$ which is in the same column as~$u_3$.
    By the definition of a flipped $(m+2)P_n$, for each $i\in\{1,3\}$, we have
    \[
        N_{G'}(u_i)\cap([m]\times[n])=N_{G'}(u'_i)\cap([m]\times[n])=N_{G''}(u'_i)\cap([m]\times[n]).
    \]
    Therefore, together with~\eqref{eq1}, we derive that
    \begin{align}
        N_{G''}(u_3)\cap([m]\times[n])=(N_{G''}(u'_1)\triangle N_{G''}(u'_3))\cap([m]\times[n]).\label{eq2}
    \end{align}
    
    Since $G''-u_3$ is the $(\mathcal{P}|_{m+1},F|_{m+1})$-flip of $(m+1)P_n$ and~$u'_1$ has no neighbors in $V(G^*)$ in $(m+1)P_n$, $w\in V(G^*)$ is adjacent to~$u'_1$ in~$G''$ if and only if ${\mathcal{P}(w)\in\mathcal{C}_F(X,X')\cup\mathcal{L}_F(X,X')}$.
    Similarly, $w\in V(G^*)$ is adjacent to~$u'_3$ in~$G''$ if and only if ${\mathcal{P}(w)\in\mathcal{C}_F(X,X')\cup\mathcal{R}_F(X,X')}$.
    Therefore, by~\eqref{eq2}, $w\in V(G^*)$ is adjacent to~$u_3$ in~$G''$ if and only if ${\mathcal{P}(w)\in\mathcal{L}_F(X,X')\cup\mathcal{R}_F(X,X')}$.
    Hence, we have
    \begin{align*}
        (N_{G''}(u'_1)\cap N_{G''}(u_3))\cap V(G^*)&=\bigcup_{Z\in\mathcal{L}_F(X,X')}Z|_m,\\
        (N_{G''}(u'_1)\setminus N_{G''}(u_3))\cap V(G^*)&=\bigcup_{Z\in\mathcal{C}_F(X,X')}Z|_m,\\
        (N_{G''}(u_3)\setminus N_{G''}(u'_1))\cap V(G^*)&=\bigcup_{Z\in\mathcal{R}_F(X,X')}Z|_m.
    \end{align*}

    Since ${X\in\mathcal{L}_F(X,X')}$, we observe that for every ${Z\in\mathcal P}$, ${(X,Z)\in F\triangle\mathcal D_{F}(X,X')}$ if and only if ${Z\in\mathcal{L}_F(X,X')\cup\mathcal{R}_F(X,X')}$.
    Similarly, since ${X'\in\mathcal{R}_F(X,X')}$, we observe that for every ${Z\in\mathcal P}$,  ${(X',Z)\in F\triangle\mathcal D_{F}(X,X')}$ if and only if $Z\in\mathcal{L}_F(X,X')\cup\mathcal{R}_F(X,X')$.
    Therefore, for every ${Z\in\mathcal{P}}$, ${(X,Z)\in F\triangle\mathcal D_{F}(X,X')}$ if and only if ${(X',Z)\in F\triangle\mathcal D_{F}(X,X')}$.
    In addition, for every ${Z\in\mathcal{C}_F(X,X')}$, neither~$(X,Z)$ nor~$(X',Z)$ is contained in ${F\triangle\mathcal D_{F}(X,X')}$.
    
    Thus, for the same~$\mathcal{P}_0$ and~$F_0$ as in the previous argument, the $(\mathcal P,F\triangle\mathcal{D}_F(X,X'))$-flip is identical to the $(\mathcal P_0,F_0)$-flip.
    Hence, $G^*$ is the $(\mathcal{P}_0|_m,F_0|_m)$-flip of~$mP_n$, and therefore $G^*$ is a $(k-1)$-flipped $mP_n$ and a pivot-minor of~$G'$.
    
    We now suppose that $\ell\geq4$.
    Let $G'':=G'\wedge u_2u_3-\{u_2,u_3\}$.
    Since each of~$u_2$ and~$u_3$ has degree~$2$ in~$G'$, by \cref{lem:shortening}, $G''$ is identical to the graph obtained from~$G'$ by replacing the path ${u_1u_2u_3u_4}$ with an edge~$u_1u_4$.
    By the inductive hypothesis for $\ell-2$, $G''$ has a $(k-1)$-flipped $mP_n$ as a pivot-minor, which is a pivot-minor of~$G'$ as well.
    This completes the proof.
\end{proof}

We have the following corollary of \cref{lem:case1,lem:case2}.

\begin{corollary}\label{cor:cases}
    For integers $m,n\geq1$, let~$\mathcal P$ be a coarsening of the column set of $(m+4)P_n$ with $\abs{\mathcal{P}}=k\geq2$.
    Then every $\mathcal P$-flip of $(m+4)P_n$ contains a $(k-1)$-flipped $mP_n$ as a pivot-minor.
\end{corollary}
\begin{proof}
    Let~$F$ be a symmetric subset of~$\mathcal{P}^2$ and let ${G=((m+4)P_n)\oplus(\mathcal{P},F)}$.
    By \cref{lem:case1}, we may assume that $(X,X)\in F$ for every set $X\in\mathcal{P}$.
    By \cref{lem:case2}, we may assume that~$G$ has no $\bigcup\mathcal{P}$-path between distinct $X,X'\in\mathcal{P}$ with $(X,X')\notin F$.

    Let $X_1,X_2\in \mathcal{P}$ be distinct sets so that a subpath~$Q_1$ of the first row of $(m+4)P_n$ from a vertex in~$X_1$ to a vertex in~$X_2$ is as short as possible.
    Then~$Q_1$ is a $\bigcup\mathcal{P}$-path in~$(m+4)P_n$ between~$X_1$ and~$X_2$.
    Note that $(X_1,X_2)\in F$, because otherwise~$Q_1$ is a $\bigcup\mathcal{P}$-path of~$G$ between~$X_1$ and~$X_2$ with $(X_1,X_2)\notin F$.
    Since $(X_1,X_1),(X_2,X_2)\in F$, both~$X_1$ and~$X_2$ are in $\mathcal{C}_F(X_1,X_2)$.
    If ${\mathcal{L}_F(X_1,X_2)\cup\mathcal{R}_F(X_1,X_2)}$ is empty, then~$G$ is a $\mathcal{P}'$-flip of $(m+4)P_n$ for ${\mathcal{P}':=(\mathcal{P}\setminus\{X_1,X_2\})\cup\{X_1\cup X_2\}}$, and therefore $G$ is a $(k-1)$-flipped $(m+4)P_n$.
    Thus, we may assume that $\mathcal{L}_F(X_1,X_2)\cup\mathcal{R}_F(X_1,X_2)$ contains a set~$X_3$.
    By swapping the roles of~$X_1$ and~$X_2$ if necessary, we may assume that ${X_3\in\mathcal{R}_F(X_1,X_2)}$.
    Note that ${X_3\notin\{X_1,X_2\}}$ as ${X_1,X_2\in\mathcal{C}_F(X_1,X_2)}$.
    This also implies that $n\ge 3$.

    By \cref{lem:one edge}, $G$ has a vertex $u\in X_2$ on the $(m+3)$-th row and a vertex $v\in X_3$ on the $(m+4)$-th row such that~$u$ and~$v$ are adjacent in~$G$.
    Let
    \begin{align*}
        G'&:=(G\wedge uv)\bigl[[m+2]\times[n]\bigr],\\
        F'&:=(F\triangle\mathcal{D}_F(X_2,X_3))|_{m+2}
    \end{align*}
    Note that~$G'$ is a pivot-minor of~$G$ and by \cref{lem:pivoting}, $G'=((m+2)P_n)\oplus(\mathcal{P}|_{m+2},F')$.

    Since~$u$ has no neighbors in~$V(G')$ in $(m+4)P_n$, a vertex $w\in V(G')$ is adjacent to~$u$ in~$G$ if and only if ${\mathcal{P}(w)\in\mathcal{C}_F(X_2,X_3)\cup\mathcal{L}_F(X_2,X_3)}$.
    Similarly, a vertex $w\in V(G')$ is adjacent to~$v$ in~$G$ if and only if ${\mathcal{P}(w)\in\mathcal{C}_F(X_2,X_3)\cup\mathcal{R}_F(X_2,X_3)}$.
    Since $X_3\in\mathcal{R}_F(X_1,X_2)$, we have $X_1\in\mathcal{L}_F(X_2,X_3)$.
    Since ${(X_2,X_2),(X_2,X_3)\in F}$, we have $X_2\in\mathcal{C}_F(X_2,X_3)$.
    Thus, by the definition of~$F'$,
    \[
        (X_1|_{m+2},X_2|_{m+2})\notin F'.
    \]
    Then~$Q_1$ is a ${\bigcup(\mathcal{P}|_{m+2})}$-path of~$G'$ between $X_1|_{m+2}$ and $X_2|_{m+2}$.
    Thus, by \cref{lem:case2}, $G'$ contains a $(k-1)$-flipped $mP_n$ as a pivot-minor, and so does~$G$.
\end{proof}

We now prove \cref{prop:preprocess}.

\begin{proof}[Proof of \cref{prop:preprocess}]
    We may assume that $n\geq2$.
    Let~$G$ be a $\mathcal{P}$-flip of $(4n-3)P_n$ for a coarsening~$\mathcal{P}$ of the column set of $(4n-3)P_n$.
    Since $\abs{\mathcal{P}}\leq n$, by applying \cref{cor:cases} at most $n-1$ times, one can find a $1$-flipped $P_n$ as a pivot-minor of~$G$.
\end{proof}

\section{Path pivot-minors from $1$-flips of a path}\label{sec:1-flip}

This section is dedicated to prove that for every integer $t\geq1$, every $1$-flip of a sufficiently long path has~$P_t$ as a pivot-minor.

\begin{proposition}\label{prop:1-flip}
    For integers $t,n\geq1$, if $n\geq3(2t^2-t-1)$, then every $1$-flip of~$P_n$ has a pivot-minor isomorphic to~$P_t$.
\end{proposition}

To prove \cref{prop:1-flip}, we will use the following lemmas.
We refer to \cref{fig:pivot0} for an illustration of \cref{lem:pivot0}.

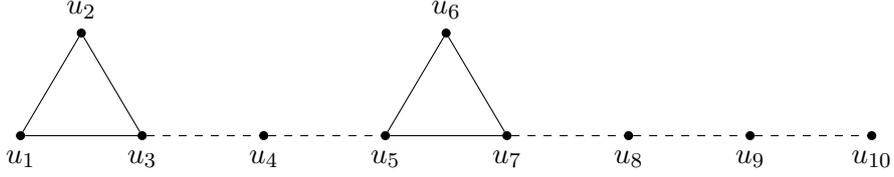
\begin{figure}
    \centering
    \begin{tikzpicture}[scale=0.8]
        \tikzstyle{v}=[circle, draw, solid, fill=black, inner sep=0pt, minimum width=3pt]

        \node[v,label=below:{$u_1$}] (u1) at (0,0) {};
        \node[v,label=above:{$u_2$}] (u2) at (1,1.7) {};
        \node[v,label=below:{$u_3$}] (u3) at (2,0) {};
        \node[v,label=below:{$u_4$}] (u4) at (4,0) {};
        \node[v,label=below:{$u_5$}] (u5) at (6,0) {};
        \node[v,label=above:{$u_6$}] (u6) at (7,1.7) {};
        \node[v,label=below:{$u_7$}] (u7) at (8,0) {};
        \node[v,label=below:{$u_8$}] (u8) at (10,0) {};
        \node[v,label=below:{$u_9$}] (u9) at (12,0) {};
        \node[v,label=below:{$u_{10}$}] (u10) at (14,0) {};

        \draw (u1)--(u2)--(u3)--(u1);
        \draw (u5)--(u6)--(u7)--(u5);
        \draw[dashed] (u3)--(u4);
        \draw[dashed] (u4)--(u5);
        \draw[dashed] (u7)--(u8);
        \draw[dashed] (u8)--(u9);
        \draw[dashed] (u9)--(u10);
    \end{tikzpicture}
    \caption{The graph~$G$ in \cref{lem:pivot0} for $P=u_1u_2\cdots u_{10}$, $a=u_4$, and $b=u_8$ where the bottom vertices are the vertices in~$X$.
    Edges between non-consecutive bottom vertices are not shown.
    We remark that ${G\wedge u_4u_8-\{u_4,u_8\}}$ is the $\{u_1,u_{10}\}$-flip of~$P_8$, which is a cycle of length~$8$.}
    \label{fig:pivot0}
\end{figure}

\begin{lemma}\label{lem:pivot0}
    Let~$G$ be the $X$-flip of a path $P$ for a set ${X\subseteq V(P)}$.
    Let $a,b\in X$ with $N_{P}[a]\cap N_{P}[b]=\emptyset$, and let~$P'$ be the path obtained from ${P-\{a,b\}}$ by making each of $N_P(a)$ and $N_P(b)$ a clique.
    Then ${G\wedge ab-\{a,b\}}$ is the $(X\triangle
    N_{P}[\{a,b\}])$-flip of~$P'$.
\end{lemma}
\begin{proof}
    Let $G':=G\wedge ab-\{a,b\}$ and let $X':=X\triangle
    N_{P}[\{a,b\}]$.
    We show that~$G'$ is the $X'$-flip of~$P'$.
    Since $N_P[a]\cap N_P[b]=\emptyset$, we observe that each vertex in $N_P[\{a,b\}]$ is adjacent to exactly one of~$a$ and~$b$ in~$G$.
    Therefore,
    \begin{align*}
        N_G(a)\triangle N_G(b)&=N_P[\{a,b\}],\\
        N_G(a)\cap N_G(b)&=X\setminus N_P[\{a,b\}].
    \end{align*}
    By \cref{lem:Oum2004}, we have $G'-N_P(\{a,b\})=G-N_P[\{a,b\}]$.
    Since $X'\setminus N_P(\{a,b\})=X\setminus N_P[\{a,b\}]$, by the definition of a flipped graph, $G'-N_P(\{a,b\})$ is the $(X'\setminus N_P(\{a,b\}))$-flip of $P-N_P[\{a,b\}]$.

    Let~$u$ and~$v$ be distinct vertices of~$G'$ with $u\in N_P(\{a,b\})$.
    To complete the proof, we show that $uv\in E(G')$ if and only if either
    \begin{itemize}
        \item $uv\notin E(P')$ and both~$u$ and~$v$ are in~$X'$, or
        \item $uv\in E(P')$ and~$u$ or~$v$ is not in~$X'$.
    \end{itemize}
    Without loss of generality, assume that $u\in N_P(a)$.
    If $v\notin X\cup N_P(\{a,b\})$, then $v\notin X'$ and it is readily seen that $uv\in E(G')$ if and only if $uv\in E(P')$.
    In addition, if $v\in X\setminus N_P(\{a,b\})$, then $v\in X'$ and by \cref{lem:Oum2004}, it holds that $uv\in E(G')$ if and only if either
    \begin{itemize}
        \item $uv\notin E(P')$ and $u\notin X$, or
        \item $uv\in E(P')$ and $u\in X$.
    \end{itemize}
    In other words, $uv\in E(G')$ if and only if either
    \begin{itemize}
        \item $uv\notin E(P')$ and both~$u$ and~$v$ are in~$X'$, or
        \item $uv\in E(P')$ and $u\notin X'$.
    \end{itemize}
    Thus, we may assume that $v\in N_P(\{a,b\})$.

    We first suppose that $v\in N_P(a)$.
    Note that $uv\in E(P')$.
    By \cref{lem:Oum2004}, we have that $uv\in E(G')$ if and only if~$u$ or~$v$ is in~$X$.
    Therefore, $uv\in E(G')$ if and only if~$u$ or~$v$ is not in~$X'$.
    
    We now suppose that $v\in N_P(b)$.
    If $uv\notin E(P)$, then by \cref{lem:Oum2004}, we have that $uv\in E(G')$ if and only if neither~$u$ nor~$v$ is in~$X$.
    Therefore, $uv\in E(G')$ if and only if both~$u$ and~$v$ are in~$X'$.
    If ${uv\in E(P)}$, then by \cref{lem:Oum2004}, we have that $uv\in E(G')$ if and only if~$u$ or~$v$ is in~$X$.
    Therefore, $uv\in E(G')$ if and only if~$u$ or~$v$ is not in~$X'$.
    This completes the proof.
\end{proof}

We will use the following lemma which is essentially from Kim and Oum~\cite[Lemma 3.3]{KO2020} written for cycles.
For integers $s\ge t\ge 0$, an \emph{$(s,t)$-path} is a graph isomorphic to the $X$-flip of~$P_s$ for some set~$X$ of~$t$ consecutive vertices of the path~$P_s$.

\begin{lemma}\label{lem:all odd sub}
    Let $s\ge t\ge 6$ be integers.
    Every $(s,t)$-path has an ${(s-2,t-6)}$-path as a pivot-minor.
\end{lemma}
\begin{proof}
    Let~$G$ be the $X$-flip of~$P_s$ for some set~$X$ of~$t$ consecutive vertices of~$P_s$.
    Let ${v_1,\ldots,v_t}$ be the vertices in $X$ in the order they appear in $P_s$.
    By \cref{lem:pivot0}, $G\wedge v_2v_{t-1}-\{v_2,v_{t-1}\}$ is an ${(s-2,t-6)}$-path.
\end{proof}

\begin{lemma}\label{lem:all odd}
    For integers $t\geq1$ and $s\ge \frac{3}{2}t+1$, the graph $\overline{P_s}$ has a pivot-minor isomorphic to~$P_t$.
\end{lemma}
\begin{proof}
    Note that $\overline{P_s}$ is an $(s,s)$-path.
    Let $r\in[5]\cup\{0\}$ such that $r\equiv s\pmod 6$.
    We prove a stronger statement that $s\ge \frac{3}{2}t$ and $r\in\{0,1,3,4\}$ imply that~$\overline{P_s}$ has a pivot-minor isomorphic to~$P_t$.
    By \cref{lem:all odd sub}, we can find a pivot-minor~$H$ of~$\overline{P_s}$ isomorphic to an $(s-2\lfloor s/6\rfloor,r)$-path.
    Let $X$ be the set of $r$ consecutive vertices of $P_{s-2\lfloor s/6\rfloor}$ such that $H$ is the $X$-flip of $P_{s-2\lfloor s/6\rfloor}$.
    Let $u_1,\dots,u_r$ be the vertices in $X$ in the order they appear in $P_{s-2\lfloor s/6\rfloor}$.
    Note that $s-2\lfloor  s/6\rfloor=s-2s/6+r/3\ge t+r/3$.

    Since a $(t,0)$-path and a $(t,1)$-path are isomorphic to $P_t$, we are done if $r\in\{0,1\}$.
    If $r\in\{3,4\}$, then $H-\{u_2,u_{r-1}\}$ is isomorphic to $P_{s-2\lfloor s/6\rfloor-\lceil r/3\rceil}$, which contains an induced subgraph isomorphic to~$P_t$, as $s-2\lfloor s/6\rfloor\geq t+\lceil r/3\rceil$.
\end{proof}

We show that for the $X$-flip~$G$ of a path, if $X$ is sufficiently large, then~$G$ has~$P_t$ as a pivot-minor.

\begin{lemma}\label{lem:reduce}
    Let~$G$ be the $X$-flip of a path $P$ and a set $X\subseteq V(P)$.
    Let $\ell\geq1$ and $m\geq0$ be integers.
    If the first~$\ell$ vertices of~$P$ are in~$X$ and ${\abs{X}\geq\ell+4m}$, then~$G$ has a pivot-minor isomorphic to $\overline{P_{\ell+m}}$.
\end{lemma}
\begin{proof}
    We proceed by induction on~$\abs{V(P)}$.
    The statement obviously holds if ${m=0}$ or $V(G)=X$.
    Thus, we may assume that ${m\geq1}$ and $V(G)\neq X$.
    By \cref{lem:shortening} and the inductive hypothesis, we may further assume that ${V(G)\setminus X}$ is an independent set of~$G$.
    We may also assume that both ends of~$P$ are in~$X$, as otherwise we can remove one end not in $X$ and apply the inductive hypothesis.
    Let $n:=\abs{V(P)}$ and let ${u_1,u_2,\ldots,u_n}$ be the vertices of~$P$ in the order they appear in~$P$ so that $u_1,u_2,\ldots,u_\ell\in X$.
    
    Let~$j$ be the smallest integer in~$[n]$ such that $u_j\notin X$.
    Note that $u_{j-1}\in X$ and $j\geq\ell+1$.
    Since $\abs{X}\geq\ell+4m>(\ell+1)+4(m-1)$, if $j\geq\ell+2$, then~$G$ has a pivot-minor isomorphic to $\overline{P_{\ell+1+m-1}}$ by the inductive hypothesis.
    Thus, we may assume that $j=\ell+1$.
    Since $V(G)\setminus X$ is an independent set of~$G$, $u_{\ell+2}$ is in~$X$.
    Since ${\abs{X}\geq\ell+4m\geq\ell+4}$, we have ${n\geq\ell+5}$, and therefore $N_{P}[u_{\ell+2}]\cap N_P[u_n]=\emptyset$.
    By \cref{lem:pivot0}, ${G':=G\wedge u_{\ell+2}u_n}-\{u_{\ell+2},u_n\}$ is the $X'$-flip of~$P_{n-2}$ for ${X':=X\triangle N_P[\{u_{\ell+2},u_n\}]}$.
    Note that ${u_i\in X'}$ for every $i\in[\ell+1]$.
    Since
    \[
        (X\cup\{u_\ell+1\})\setminus\{u_{\ell+2},u_{\ell+3},u_{n-1},u_{n}\}\subseteq X',
    \]
    we have ${\abs{X'}\geq\ell+1+4m-4=(\ell+1)+4(m-1)}$.
    By the inductive hypothesis, $G'$ has a pivot-minor isomorphic to $\overline{P_{\ell+1+m-1}}$, and therefore $G$ has a pivot-minor isomorphic to $\overline{P_{\ell+m}}$.
\end{proof}

\cref{lem:reduce} has the following corollary.

\begin{corollary}\label{cor:reduce}
    Let~$G$ be the $X$-flip of a path~$P$ for a set $X\subseteq V(P)$.
    For an integer $t\geq1$, if $\abs{X}\geq4t-3$, then~$G$ has a pivot-minor isomorphic to~$\overline{P_t}$.
\end{corollary}
\begin{proof}
    By taking a subpath, we may assume that~both ends of $P$ are in~$X$.
    Since
    \(
        \abs{X}\geq4t-3=1+4(t-1),
    \)
    by \cref{lem:reduce} for ${\ell=1}$ and ${m=t-1}$, the statement holds.
\end{proof}

We now prove \cref{prop:1-flip}.

\begin{proof}[Proof of \cref{prop:1-flip}]
    We may assume that~$G$ has no~$P_t$ as an induced subgraph.
    Observe that every connected induced subgraph of~$G$ with at most one vertex in~$X$ is an induced path.
    Since~$G$ has no~$P_t$ as an induced subgraph, this implies that every component of $G-X$ has at most ${t-2}$ vertices.

    Suppose that $\abs{X}\leq6t+2$.
    Since $G-X=P_n-X$ has at most $\abs{X}+1\leq6t+3$ components,
    \[
        n\leq (6t+3)(t-2)+(6t+2)=3(2t^2-t-1)-1,
    \]
    a contradiction.
    Hence, $\abs{X}\geq6t+3$.

    Let $s:=\lceil\frac{3}{2}t+1\rceil$.
    Note that
    \[
        4s-3=4\left\lceil\frac{3}{2}t+1\right\rceil-3\leq4\left(\frac{3}{2}t+\frac{3}{2}\right)-3=6t+3\leq\abs{X}.
    \]
    Thus, by \cref{cor:reduce,lem:all odd}, $G$ has a pivot-minor isomorphic to~$P_t$.
\end{proof}

\section{Completing the proof of the main theorem}\label{sec:proof}

We now prove \cref{thm:main}.

\begin{proof}[Proof of \cref{thm:main}]
    Let $n:=3(2t^2-t-1)$ and let $s:=4n-3$.
    For the function~$g$ in \cref{thm:Mahlmann25}, we set~$f(t)$ as~$g(s)$.
    We show that every graph~$G$ of rank-depth at least~$f(t)$ has~$P_t$ as a pivot-minor.

    By \cref{thm:Mahlmann25}, $G$ has an induced subgraph~$H$ isomorphic to $K_s\tri K_s$, $K_s\tri\overline{K_s}$, $\overline{K_s}\tri\overline{K_s}$, or a flipped $sP_s$.
    We may assume that~$H$ is not isomorphic to $K_s\tri K_s$.
    If~$H$ is isomorphic to $K_s\tri\overline{K_s}$ or $\overline{K_s}\tri\overline{K_s}$, then by \cref{lem:second,lem:third}, $H$ has a pivot-minor isomorphic to~$P_t$.
    Hence, we may assume that~$H$ is a flipped $sP_s$.

    By \cref{prop:preprocess}, $H$ has a pivot-minor~$H'$ isomorphic to a $1$-flip of~$P_n$.
    By \cref{prop:1-flip}, $H'$ has a pivot-minor isomorphic to~$P_t$.
\end{proof}

\paragraph{Acknowledgments.}
The authors would like to thank Donggyu Kim and Rose McCarty for bringing the result of M\"{a}hlmann~\cite{Mahlmann25arxiv} to our attention.
This research was initiated at the Vertex-Minor Workshop in Jeju organized by the IBS Discrete Mathematics Group in 2023.

\providecommand{\bysame}{\leavevmode\hbox to3em{\hrulefill}\thinspace}
\providecommand{\MR}{\relax\ifhmode\unskip\space\fi MR }
\providecommand{\MRhref}[2]{%
  \href{http://www.ams.org/mathscinet-getitem?mr=#1}{#2}
}
\providecommand{\href}[2]{#2}

\end{document}